\documentclass[reqno]{amsart}
\usepackage{amsmath,amsfonts,amssymb,graphics,graphicx,latexsym,amscd, subfigure}

\newtheorem{thm}{Theorem}[section]

\newtheorem{prop}[thm]{Proposition}

\newtheorem{lemma}[thm]{Lemma}
\newtheorem{defn}[thm]{Definition}

\newtheorem{exa}[thm]{Example}
\newtheorem{nota}[thm]{Notation}

\newcommand{\bbR}{\mathbb{R}}

\newcommand{\bbT}{\mathbb{T}}
\newcommand{\bbC}{\mathbb{C}}
\newcommand{\bbZ}{\mathbb{Z}}
\newcommand{\bbH}{\mathbb{H}}
\newcommand{\ip}{\cdot}
\newcommand{\cur}{\Theta}
\newcommand{\s}{u}
\newcommand{\TR}{\mathcal{R}}
\newcommand{\oh}{O\left(\frac{1}{H}\right)}

\newcommand{\oz}{O(0)}
\newcommand{\oum}{O(-1)}
\newcommand{\oo}{O(-2)}
\newcommand{\ooo}{O(-3)}
\newcommand{\oooo}{O(-4)}

\DeclareMathOperator{\id}{id}

\DeclareMathOperator{\Hess}{Hess}
\DeclareMathOperator{\Det}{det}

\DeclareMathOperator{\Tr}{tr}
\DeclareMathOperator{\sg}{sgn}

\newcommand{\ov}{\overline}

\newcommand{\p}{{\partial}}

\newcommand{\pp}{\perp}
\newcommand{\lt}{\mathcal{L}^2}

\begin{document}
\title[Curvature of scalar-flat toric K\"ahler metrics]
{Curvature of scalar-flat K\"ahler metrics on non-compact symplectic toric $4$-manifolds}

\author{Rosa Sena-Dias}
\begin{thanks}
{Partially supported by the Funda\c{c}\~{a}o para a Ci\^{e}ncia e a Tecnologia
(FCT/Portugal).}\
\end{thanks}
\email{ rsenadias@math.ist.utl.pt}

\date{}

\begin{abstract}
In this paper, we show that the complete scalar-flat K\"ahler metrics constructed in \cite{as} on strictly unbounded toric $4$-dimensional orbifolds have finite $\lt$ norm of the full Riemannian tensor. In particular, this answers a question of Donaldon's from \cite{d2} on the corresponding Generalized Taub-NUT metric on $\bbR^4$.  This norm is explicitly determined when the underlying toric manifold is the minimal resolution of a cyclic singularity of the form $\bbC^2/\bbZ^k$. In the Ricci-flat case corresponding to gravitational instantons, this recovers a recent result in \cite{al}.

\end{abstract}

\maketitle

\section{Introduction}

The purpose of this note is to show that the metrics constructed in \cite{as} have finite energy as conjectured by Donaldson in \cite{d2}. Namely we prove the following proposition:
\begin{thm}\label{main}
Let $X$ be a strictly unbounded K\"ahler toric orbifold of real dimension $4$ endowed with one of the scalar- flat K\"ahler toric metric constructed in \cite{as}. The Calabi functional  
\[
\int_X|\TR|^2,
\]
where $\TR$ is the Riemann curvature tensor, is finite. 
\end{thm}
An unbounded toric $4$-manifold is roughly a non-compact symplectic $4$-manifold with a $\bbT^2$ action which is effective and Hamiltonian. Such a manifold is said to be strictly unbounded if the moment map image has non-parallel unbounded edges. The simplest examples of such a manifolds are $\bbR^4$ and all its blow-ups. We will give a precise definition ahead.

The asymptotic behavior of the metrics in our main theorem was investigated in \cite{as} in order to prove that the constructed metrics were complete. The results there were not sufficient to show that the Calabi functional of these metrics is finite as conjectured by Donaldson in \cite{d2}. In this note, we use the explicit construction of \cite{as} to prove Theorem \ref{main}.  In fact our proof of theorem \ref{main} is essentially direct and we basically proceed by actually calculating 
\[
\int_X|\TR|^2.
\]
In section \ref{example} we give an example where we carry out the calculation for the the Calabi functional explicitly. The general case is not much harder and is reduced to determining the integral of a rational function of a single real variable. Given any specific example of a strictly unbounded K\"ahler toric orbifold of real dimension $4$ endowed with one of the scalar- flat K\"ahler toric metric from \cite{as} we can carry out the computation of its Calabi functional by integrating a rational function of a single real variable. The degree of the polynomials in the denominator and numerator increases with the number of edges of the moment map image so the calculation becomes more complicated, but can be carried out.

There are still interesting open questions about the asymptotic behavior of the metrics from \cite{as}. For example, for parameters $\nu=(\alpha,\beta)$ for which $\alpha/\beta$ is irrational it is unclear if the corresponding metrics fall into the ALF category. This is a question we plan to address in subsequent work.  

Also we conjecture that any scalar-flat K\"ahler toric metric on the strictly unbounded toric 4-manifold $X$ is a so-called Donaldson generalized Taub-NUT metric corresponding to an admissible parameter $\nu \in \bbR^2$ as constructed in \cite{as}. If this is true, the above proposition shows that the Calabi functional of any scalar- flat K\"ahler toric metric on a strictly unbounded $4$-dimensional orbifold is finite.

We will assume that the reader is somewhat familiar with the construction in \cite{as} but we will quickly give some background on unbounded toric 4-orbifolds and scalar-flat K\"ahler toric metrics on them. 

The paper is organized as follows. We first give a very brief review of the results in \cite{as}. Then in section \ref{chern_classes} we explain how to calculate
\[
\int_X|\TR|^2,
\]
in terms of 
\[
\int_X c_1^2 \quad \text{and} \, \int_X c_2.
\]
This is probably well understood by most readers but we recall it here for the sake of completeness. We describe a set of convenient coordinates called isothermal coordinates which were also used in \cite{as} and explain how to express the Hermitian metric on the tangent bundle corresponding to the scalar-flat K\"ahler metric in these coordinates. In sections \ref{c_1} and \ref{c_2} we show that by using isothermal coordinates our calculations reduce to the calculation of a pair of integrals on $\bbH$. We also show that the involved integrals are finite. These sections are the heart of the paper. Section \ref{c_2} has rather cumbersome calculations but section \ref{c_1} already gives a flavor of what is involved in the proof of our main theorem and yields relatively simple results. Section \ref{example} particularizes to the case when the underlying  toric manifold is the minimal resolution of a cyclic singularity. Calculations simplify in this setting making it possible to carry out the integrals whose finiteness was shown in sections \ref{c_1} and \ref{c_2}. 

\vspace{.5cm}
\noindent \textbf{Acknowledgments:} Most of this work was carried out in 2010 while I was visiting Imperial College. I would like to take the opportunity to thank the Mathematics department there for its hospitality. In particular I would like to thank Simon Donaldson,  Andr\'e Neves and Richard Thomas. I would also like to thank Claude Lebrun for his encouragement and for reviving my own interest in these calculations.

\section{background}
For a more leisurely treatment of unbounded toric manifolds the reader should refer to \cite{as}. We begin by recalling the definition of a K\"ahler toric orbifold
\begin{defn} \label{def:torb}
A K\"ahler toric $4$-orbifold is a connected $4$-dimensional
K\"ahler orbifold $(X,\omega,J)$ equipped with an effective, Hamiltonian,
holomorphic action of the standard (real) $2$-torus $\bbT^2 = \bbR^2/2\pi\bbZ^2$, such that
the corresponding moment map $\phi : X \to \bbR^2$,  is proper onto its convex image $P=\phi(X)\subset \bbR^2$.
\end{defn}
We will actually work with a subclass of non-compact K\"ahler toric orbifolds so we make the following definition
\begin{defn}
A K\"ahler toric $4$-orbifold is said to be \emph{unbounded} if its
moment polygon is unbounded and has a finite number of vertices.

A K\"ahler toric $4$-orbifold is said to be \emph{strictly unbounded} if it is unbounded and its moment polygon has non-parallel unbounded edges.
\end{defn}

We showed in \cite{as} that any unbounded K\"ahler toric $4$-manifolds is biholomorphic to a finite blow up of minimal resolutions of cyclic singularities. In the same way unbounded K\"ahler toric $4$-orbifolds are a finite blow up of $\bbC^2/\bbZ^k$ for some integer $k$ where  $\bbC^2/\bbZ^k$ has an isolated singularity at the origin. 

Let us fix some notation. As we have seen in \cite{as}, if $X$ is a strictly unbounded toric $4$-dimensional orbifold, the image of $\phi$ is a convex unbounded polygonal region in $\bbR^2$ which we will denote by $P$. The region $P$ has a finite number of edges labeled $1,\cdots,d$. Each of the edges has a primitive normal vector in $\bbZ^2$ $\nu_i, i=1,\cdots, d$ and $P$ is given by
\[
P=\{x\in \bbR^2: x\ip\nu_i\geq \lambda_i, i=1,\cdots,d\}
\]
for some real numbers $\lambda_i$. There are exactly two unbounded edges on $P$. The first edge normal to $\nu_1$ and the last edge normal to $\nu_d$. The fact that $X$ is strictly unbounded ensures that these are not parallel. See \cite{as} for more details.  We recall the main results in \cite{as} 
\begin{thm}[Abreu,Sena-Dias]
Let $X$ be a strictly unbounded toric 4-manifold with moment map image $P$
\[
P=\{x\in \bbR^2: x\ip\nu_i\geq \lambda_i, i=1,\cdots,d\}.
\] 
 Let $\nu = (\alpha,\beta)$ be a vector in $\bbR^2$ such that
\[
\det(\nu_1,\nu),\,\,\det(\nu_d,\nu)>0.
\]
Then $X$ admits a complete scalar-flat K\"ahler toric metric called Donaldson generalized Taub-NUT metric with parameter $\nu$.
\end{thm}
Theorem \ref{main} applies  to these metrics.

For later reference we recall the definition of symplectic coordinates on $X$. The reader who is not familiar with these coordinates can look at \cite{a1} or \cite{a2}. Alternatively there is some background in \cite{as}. 

Let $X_0=\phi^{-1} (\text{int} (P))$. The torus $\bbT^2$ acts freely on $X_0$ and 
\[
X_0\simeq\text{int} (P)\times \bbT^2.
\]
There are (symplectic) coordinates on $X_0$ that realize the above identification. We denote these by $(x_1,x_2,\theta_1,\theta_2)$ where $(x_1,x_2)$ take values in $P$ and $(\theta_1,\theta_2)$ are the angle coordinates on the torus $\bbT^2$. There is also a so-called symplectic potential associated to each K\"ahler toric metric on $X$. It is a smooth function on $\text{int} (P)$ which we will denote by $\s$. It completely determines the metric. We set $\xi_1=u_{x_1}$ and $\xi_2=u_{x_2}$. As it turns out $(\xi_1,\xi_2,\theta_1,\theta_2)$ give complex coordinates on $X$. 

Another useful set of coordinates on $X_0$ is given as follows. The toric K\"ahler metric determines a metric on the real locus of $X$. The metric induced on this real  2-dimensional manifold admits a set of isospectral coordinates which we call $(H,r)$ where
$$
r=(\Det(\Hess u))^{-\frac{1}{2}}
$$
These were described by Donaldson in \cite{d2} and used in \cite{as}. They play a very important role here as well. It turns out that they define a diffeomorphism between $\text{int} (P)$ and $\bbH=\{(H,r):r>0\}$ which we will denote by $\gamma$. 

We start by showing that the finiteness of the first and second Chern classes of $X$ imply the finiteness of the Calabi functional.

\section{The Calabi functional and Chern classes}\label{chern_classes}
\subsection{Characteristic classes on $X$}
The K\"ahler metric determines a Hermitian connection $\theta$ on the complex fiber bundle $TX$. On a coordinate patch, using a local frame, $\theta$ can be interpreted as a $2$ by $2$ matrix of $1$-forms.  As in the compact case, such a connection has a curvature form which we denote by $\Theta$. It is given by
\[
\Theta=d\theta+\theta\wedge  \theta.
\]
Locally, the curvature $\Theta$ can be seen as a $2$ by $2$ matrix of $2$-forms using a local frame for $TX$. Changing frames on $TX$ changes the connection form to an adjoint matrix and therefore, using ad-invariant polynomials (i.e. polynomials that are invariant under the adjoint action of $U(2)$ on the space of $2$ by $2$ matrices), there are globally well defined $2$-forms to be extracted from the the curvature $\Theta$. These are called characteristic forms. See \cite{w} for a more precise definition. 

We will use the following characteristic forms.
\begin{itemize}
\item $c_1=\Tr \Theta$, where $\Tr$ denotes the trace, is the first Chern class of $X$,
\item  $c_2=\det\Theta$, where $\det$ denotes the determinant, is the second Chern class of $X$,
\item $p=\Tr (\Theta \wedge \Theta)$ is the Pontryagin class of $X$.
\end{itemize}
On a $2$-dimensional complex manifold these are related because the space of ad-invariant polynomials of $2$ by $2$ matrices is $2$-dimensional. We will use the following very easy lemma.
\begin{lemma}\label{2dimcc}
The forms $c_1^2$ and $p$ determine the form $c_2$.
\end{lemma}
\begin{proof}
Using a local frame write the curvature form as
\[
\cur=
\begin{pmatrix}
\cur_{11}&\cur_{12}\\
\cur_{21}&\cur_{22}
\end{pmatrix}.
\]
Then 
\[
c_1^2=(\cur_{11}+\cur_{22})\wedge (\cur_{11}+\cur_{22})=\cur_{11}^2+\cur_{22}^2+2\cur_{11}\wedge \cur_{22},
\]
and
\[
p=\cur_{11}^2+\cur_{12}\wedge\cur_{21}+\cur_{21}\wedge\cur_{12}+\cur_{22}^2=
\cur_{11}^2+\cur_{22}^2+2\cur_{12}\wedge\cur_{21}.
\]
As for the $c_2$ it is given by 
\[
c_2=\cur_{11}\wedge \cur_{22}-\cur_{21}\wedge\cur_{12},
\]
and therefore $c_1^2-p=2c_2$.
\end{proof}

\subsection{The Calabi functional and characteristic classes}
In fact we have introduced characteristic classes because our main result turns out to be equivalent to the integrability of the characteristic classes of $TX$. The following lemma holds.
\begin{lemma}
Let $X$ be a non-compact 4-dimensional scalar-flat K\"ahler manifold. Then if 
\[
\int_X c_1^2, \,\,\, \int_X p
\]
are finite, $X$ has finite energy i.e. 
\[
\int_X |\TR|^2
\]
is finite.
\end{lemma} 
\begin{proof}
On a K\"ahler manifold, following the notation in \cite{b}, we can write the Riemann curvature tensor norm as linear combination of $s^2$, $|\rho_0|^2$ and $|B_0|^2$ where 
\begin{itemize}
\item $s$ is the scalar curvature,
\item $\rho_0$ is the primitive part of the Ricci tensor, 
\item and $|B_0|^2$ is the primitive part of the so-called $B$ tensor (see \cite{b}, pag. 77). 
\end{itemize}
Now consider the following Apte identities (see \cite{b}, pag. 80)
\begin{equation}
\frac{1}{2}\wedge^2 c_2=\frac{1}{8\pi^2}\left(\frac{s^2}{12}-|\rho_0|^2+|B_0|^2\right),
\end{equation}
\begin{equation}
\frac{1}{2}\wedge^2 c_1=\frac{1}{4\pi^2}\left(\frac{s^2}{8}-|\rho_0|^2\right).
\end{equation}
where $\wedge$ is contraction by $\omega$ i.e. the formal adjoint of wedging with $\omega$. Our metrics are all scalar-flat i.e. $s=0,$ so we see from the above formulas, that to show that $\int|\TR|^2$ is finite it is enough to show that $\int c_2$ and $\int c_1^2$ are finite. Or, using the above lemma \ref{2dimcc}, we see that to show that the Calabi functional is finite, it is enough to see that  $\int c_1^2$ and  $\int p$ are finite.
 \end{proof}

We will work in the coordinates $(H,r)$ introduced in \cite{as}. These are isospectral coordinates for the metric on the real locus of $X$. We have 
\[
r=(\det \Hess u)^{-\frac{1}{2}}.
\]
It is not hard to see that $r$ is harmonic for the restriction of the scalar-flat metric on $X$ to the real locus of $X$. We define $H$ to be its harmonic conjugate. In these coordinates the scalar-flat metric is given by specifying an $\bbR^2$ valued function $\xi(H,r)$ which axi-symmetric and harmonic when seen as function in $\bbR^3$. Here $H$ is the third coordinate in $\bbR^3$ and $r$ is the distance to $H$-axis. In \cite{as} we study a specific choice for $\xi$ given a polytope. Namely 
\begin{displaymath}
\xi = \nu_1\log(r)+\frac{1}{2}\sum_{i=1}^{d-1}(\nu_{i+1}-\nu_{i})\log\left( H_i+\rho_i\right)+\nu H
\end{displaymath}
where 
\begin{itemize}
\item the $\nu_i$ denote the normals to the facets of the moment polytope of $X$, 
\item the vector $\nu$ is such such that 
\[
\Det(\nu,\nu_1),\Det(\nu,\nu_d) \geq 0,
\]
\item Given real number $0<a_1<\cdots<a_{d-1}$ appropriately chosen in terms of the moment polytope, $H_i$ is simply $H+a_i$ for all $i=1, \cdots, d-1$
\item $\rho_i$ is simply $\sqrt{r^2+(H+a_i)^2}$ for all $i=1, \cdots, d-1$.
\end{itemize}
Given the vector $\xi$, the coordinates $(H,r)$ give a bijective proper identification of the moment polytope of $X$ with the upper half plane $\bbH$. See \cite{as} for more details. 

We set up some notation that will be useful later on. Let $(\alpha_i,\beta_i)$ denote the components of the normal $\nu_i$ to the $i$th facet of $P$ where $i=1,\cdots,d$.  We have
\[
D\xi=\begin{pmatrix}
\sum_{i=1}^{d-1}\frac{(\alpha_{i+1}-\alpha_{i})}{2\rho_i}+\alpha&\frac{1}{r}\left({\alpha_1}+\frac{1}{2}\sum_{i=1}^{d-1}\left(1-\frac{H_i}{\rho_i}\right)(\alpha_{i+1}-\alpha_{i})\right)  \\
\sum_{i=1}^{d-1}\frac{(\beta_{i+1}-\beta_{i})}{2\rho_i}+\beta&\frac{1}{r}\left({\beta_1}+\frac{1}{2}\sum_{i=1}^{d-1}\left(1-\frac{H_i}{\rho_i}\right)(\beta_{i+1}-\beta_{i})\right)  
\end{pmatrix}.
\]    
Let
\begin{equation}\label{f}
f=\nu_1+\frac{1}{2}\sum_{i=1}^{d-1}\left(1-\frac{H_i}{\rho_i}\right)(\nu_{i+1}-\nu_{i})
\end{equation}
and 
\begin{equation}\label{g}
g=\nu+\sum_{i=1}^{d-1}\frac{(\nu_{i+1}-\nu_{i})}{2\rho_i},
\end{equation}
denote the components of $g$ by $(g_a,g_b)$ and the components of $f$ by $(f_a,f_b)$ then
\begin{equation}\label{Dxi}
D\xi=\begin{pmatrix}
g_a&\frac{f_a}{r}  \\
g_b&\frac{f_b}{r}
\end{pmatrix}.
\end{equation}  
Said in another way, the columns of $D\xi$ are precisely $g$ and $\frac{f}{r}$. This implies that $V=r\Det D\xi$ is given by
$$
V=\Det (g,f).
$$
We also have that  
\begin{equation}\label{inversa_Dxi}
D\xi^{-1}=\frac{1}{V}\begin{pmatrix}
f_b&-{f_a}  \\
-rg_b&rg_a
\end{pmatrix}.
\end{equation} 
Said in another way, the lines of $D\xi^{-1}$ are $f^\pp=(f_b,-f_a)$ and $-g^\pp=(-g_b,g_a)$. 

Let $R$ be a positive large number and $B(R)$ be the pre-image via the moment map in $X$ of the ball of radius $R$ in $\bbH$ i.e.
\[
B(R)=\phi^{-1}\gamma^{-1}(\{(H,r)\in \bbH: \rho=\sqrt{r^2+H^2}<R\})
\]
where $\phi$ denotes the moment map. We are actually using $(H,r)$ to identify the moment polytope with $\bbH$. If we show that the quantities
\[
\int_{B(R)}c_1^2
\]
and
\[
\int_{B(R)}p
\]
are bounded as $R$ tends to infinity, this will finish our proof. We will often write $B(R)$ for both the set $\{(H,r)\in \bbH: \rho=\sqrt{r^2+H^2}<R\}$ and its pre-image via the moment map $\gamma\circ\phi$.

The boundary of $B(R)$ in $\bbH$ has $2$ distinct subsets that we will treat separately in calculations. Namely 
$$
\p B(R)=L_R \cup C_R
$$
where 
\begin{itemize}
\item the set $L_R$ is given by $L_R=\{(H,0): -R<H<R\}$,
\item  the set $C_R$ is given by $C_R=\{(H,r): \sqrt{r^2+H^2}=R\}$.
\end{itemize}

\subsection{The Hermitian metric in the toric setting}
Consider a non-compact toric manifold $X$ endowed with a K\"ahler toric metric. This K\"ahler metric determines a Hermitian metric on $TX$ and we want to give an expression for this metric in a local frame using the symplectic potential.
\begin{lemma}\label{hermitian_metric}
Let $X$ be a toric manifold with a K\"ahler toric metric. Consider complex coordinates in the usual way in the pre-image via the moment map of the interior of the moment map image of $X$. In this frame, the Hermitian metric on $TX$ is simply $2\Hess u$ where $u$ is the symplectic potential associated to the K\"ahler toric metric. The induced connection is given by
\[
\theta=-\partial \Hess u \Hess^{-1} u.
\]
\end{lemma}
\begin{proof}
We will use isospectral coordinates $(H,r)$ on $X$ to prove this lemma. Let $P$ be the moment map image of $X$ in $\bbR^2$. Let $(\xi_1, \xi_2)$ denote coordinates on $\bbR^2$. Over the interior of $P$, the torus action on $X$ is free and $X$ admits complex coordinates $(z_1,z_2)$ given by
\[
z_a=\xi_a+i\theta_a
\]
for $a=1,2$.
In the above formula $\xi_a=u_{x_a}$ can be expressed as functions of the coordinates $(H,r)$ by the formulas in \cite{as} and $(\theta_1,\theta_2)$ refer to the angle coordinates on the tori above each point in $\text{int}(P)$. For $a=1,2$, we can write
\[
\frac{\partial}{\partial z_a}=\frac{\partial H}{\partial \xi_a}\frac{\partial}{\partial H}+\frac{\partial r}{\partial \xi_a}\frac{\partial}{\partial r}-i\frac{\partial}{\partial \theta_a}
\]
In the coordinates $(x_1,x_2,\theta_1,\theta_2)$ the metric is given by
\begin{equation}\nonumber
\begin{bmatrix}
\phantom{-}\Hess u & \vdots & 0\  \\
\hdotsfor{3} \\
\phantom{-}0 & \vdots & \Hess^{-1} u
\end{bmatrix},
\end{equation}
where $u$ is the symplectic potential of the metric. Since $(H,r)$ are isospectral coordinates for the restriction of the metric to the polytope, in the coordinates $(H,r,\theta_1,\theta_2)$ the metric is given by
\begin{equation} \label{metric_rH}
\begin{bmatrix}
\phantom{-}\begin{pmatrix}
 V&0\\
  0&V
  \end{pmatrix}
   & \vdots & 0\  \\
\hdotsfor{3} \\

\phantom{-}0 & \vdots & \Hess^{-1}\s
\end{bmatrix}
\end{equation}
where $V=r \det D\xi$. Here, $D\xi$ denotes the matrix of the derivatives of $(\xi_1,\xi_2)$ as a function of $(H,r)$. The fiber bundle $TX$ is locally trivialized by $(\frac{\partial}{\partial z_1},\frac{\partial}{\partial z_2})$ and using expression (\ref{metric_rH}) for the metric, we can see that the Hermitian metric in this trivialization is given by 
\[
\left( \frac{\partial}{\partial z_a}, \frac{\partial}{\partial z_b}\right)=V\left(\frac{\partial r}{\partial \xi_a}\frac{\partial r}{\partial \xi_b}+ \frac{\partial H}{\partial \xi_a}\frac{\partial H}{\partial \xi_b} \right)+ u^{ab},
\]
with $a,b \in \{1,2\}$. Here $u^{ab}$ denotes the entries of the matrix $\Hess^{-1} u$, as usual. The matrix of the Hermitian metric in the frame $(\frac{\partial}{\partial z_1},\frac{\partial}{\partial z_2})$ is given by
\[
V (D\xi^{-1})^tD\xi^{-1}+\Hess^{-1} u.
\]
Next we prove a lemma that will be extremely useful to us ahead. 
\begin{lemma}\label{Hess}
Let $g$ be a K\"ahler toric metric on a symplectic toric four-fold $X$ with symplectic potential $u$ on its moment polytope $P$. Let $(H,r)\in \bbH$ be isothermal coordinates for $g$ and $\xi=(u_{x_1},u_{x_2})$. Then 
\[
\Hess u(x)=\frac{D\xi D\xi^{t}}{V}(H,r)
\]
where $V=r\det{D\xi}$ and $(H,r)$ corresponds to $x=(x_1,x_2)$ under the identification $\mathbb{H}\simeq P$.
\end{lemma}
\begin{proof}
Let $\eta=(u_{x_1},u_{x_2})$ and $\mu:\mathbb{H}\rightarrow P$ be the moment map in $(H,r)$ coordinates. We have that $\eta=\xi\circ\mu^{-1}$ and
\[
\Hess u=D\eta=D\xi D\mu^{-1}
\]
Now as we have seen in \cite{as}
\[
D\mu=r\begin{pmatrix}
\xi_{2,r} & -\xi_{2,H}  \\
-\xi_{1,r} & \xi_{1,H}  
\end{pmatrix},
\] 
hence 
\[
D\mu^{-1}=\frac{D\xi ^t}{r\det D\xi}
\]
and the result follows.
\end{proof}
From the above lemma it follows that the Hermitian metric is simply $2\Hess^{-1} u$. This implies that the Hermitian connection matrix $\theta$ in our frame is given by
\[
\theta=\partial \Hess^{-1} u \Hess u=-\partial \Hess u \Hess^{-1} u.
\]
as stated.
\end{proof}

\section{The first Chern class}\label{c_1}
In this section we start by showing that 
\[
\int_X c_1^2
\]
is finite. The matrix of the Hermitian metric on $TX$ in a local frame over $X_0$ is given by $2\Hess^{-1} u$ (see lemma \ref{hermitian_metric}) and thus
\[
c_1=\frac{1}{2}\ov\p\p \log \det \Hess^{-1} u.
\]
From \cite{d1} we know that
\[
r=(\det \Hess u)^{-\frac{1}{2}},
\]
therefore
\[
c_1=d\p \log r.
\] 
We can write
\[
\int_{B(R)}c_1^2=\int_{B(R)}d\left(\p \log r\wedge d \p\log r\right).
\]
Using Stokes theorem this becomes
\[
\int_{B(R)}c_1^2=\int_{\p B(R)}\frac{\p r}{r}\wedge d\frac{\p r}{r}=\int_{\p B(R)}\frac{\p r\wedge d\p r}{r^2}.
\]
Now we need to calculate $\p r$ using the complex structure associated to our K\"ahler metric. This calculation will be important to us ahead so we state it as a lemma.
\begin{lemma}\label{partial}
Let $X$ be a strictly unbounded toric manifold with a scalar-flat K\"ahler metric given via a symplectic potential $u$. Let $(H,r)$ be isospectral coordinates for the metric on the real locus of $X$ with
\[
r=(\det \Hess u)^{-\frac{1}{2}},
\]
and $(\theta_1,\theta_2)$ be angle coordinates on $\bbT^2$. Write $\xi=(u_{x_1},u_{x_2})$ where $x=(x_1,x_2)$ are the coordinates on $P$. Then
\[
\begin{pmatrix}
\p H\\
\p r
\end{pmatrix}=
\begin{pmatrix}
dH\\
dr
\end{pmatrix}+iD\xi^{-1} 
\begin{pmatrix}
d\theta_1\\
d\theta_2
\end{pmatrix},
\]
where $D\xi$ denotes the matrix of derivatives of $\xi$ as a function of $(H,r)$.
\end{lemma}
\begin{proof}
We have 
\[
\begin{pmatrix}
dH\\
dr
\end{pmatrix}=\frac{\p (H,r)}{\p x } \begin{pmatrix}
dx_1\\
dx_2
\end{pmatrix}
\]
where 
\[
\frac{\p (H,r)}{\p x}=\begin{pmatrix}
\frac{\p H}{\p x_1}&\frac{\p H}{\p x_2}\\
\frac{\p r}{\p x_1}&\frac{\p r}{\p x_2}
\end{pmatrix}.
\]
In symplectic coordinates $(x_1,x_2, \theta_1,\theta_2)$ the complex structure can be written as 
\[
\begin{bmatrix}
\phantom{-}
   0& \vdots & -\Hess^{-1}\s\  \\
\hdotsfor{3} \\

\phantom{-}\Hess \s & \vdots & 0
\end{bmatrix}
\]
so that 
\[
J\begin{pmatrix}
dx_1\\
dx_2
\end{pmatrix}=-\Hess^{-1}\s\begin{pmatrix}
d\theta_1\\
d\theta_2
\end{pmatrix}
\]
and
\[
J\begin{pmatrix}
dH\\
dr
\end{pmatrix}=\frac{\p (H,r)}{\p x }J \begin{pmatrix}
dx_1\\
dx_2
\end{pmatrix}=-\frac{\p (H,r)}{\p x } \Hess^{-1} \s\begin{pmatrix}
d\theta_1\\
d\theta_2
\end{pmatrix}.
\]
We will use the following simple facts
\begin{enumerate}
\item
\[
\Hess u=\frac{D\xi D\xi^t}{V}
\]
\item
\[
D\gamma=\frac{\p (H,r)}{\p x }=\frac{D\xi^t}{V}.
\]
\end{enumerate}
The first is proved in lemma \ref{Hess}, the second comes up in the proof of lemma \ref{Hess}. We  have
\[
\frac{\p (H,r)}{\p x } \Hess^{-1} \s=\frac{D\xi^t}{V}V(D\xi^{-1})^tD\xi^{-1}=D\xi^{-1}.
\]
and
\[
J\begin{pmatrix}
dH\\
dr
\end{pmatrix}=-D\xi^{-1}\begin{pmatrix}
d\theta_1\\
d\theta_2
\end{pmatrix}.
\]
Because
\[
\begin{pmatrix}
\p H\\
\p r
\end{pmatrix}=
\begin{pmatrix}
dH\\
dr
\end{pmatrix}-iJ
\begin{pmatrix}
dH\\
dr
\end{pmatrix},
\]
the result follows.
\end{proof}
The above lemma gives a more tractable expression for $\int_X c_1^2$. We have the following
\begin{lemma}\label{c1_as_det}
We have
$$
\int_Xc_1^2=\lim_{R\rightarrow \infty}\int_{\partial B(R)}\frac{\Det(dg,g)}{\Det^2(g,f)}.
$$
\end{lemma}
\begin{proof}
We start by showing that 
$$
\int_Xc_1^2=\lim_{R\rightarrow \infty}\int_{\partial B(R)}\frac{adb-bda}{r^2},
$$
where $(a,b)$ is the last line of the matrix $D\xi^{-1}$. We know that $\int_Xc_1^2$ is given by
\[
\lim_{R\rightarrow \infty}\int_{\p B(R)}\frac{\p r \wedge d\p r}{r^2} 
\]
First note that from lemma \ref{partial}, if we denote the elements in the second line of the matrix $D\xi^{-1}$ by $a$ and $b$ we have
\[
\p r=dr+i(ad\theta_1+bd\theta_2).
\]
It follows that
\[
d\p r=i\left(da\wedge d\theta_1+db\wedge d\theta_2\right),
\
\]
and
\[
\p r \wedge d\p r=-(adb-bda)\wedge d \theta_1\wedge d\theta_2+idr\wedge (da\wedge d\theta_1+db\wedge d\theta_2).
\] 
Since we are integrating over the pre-image of a curve in $P$, any multiple of the form $dr\wedge(da\wedge d\theta_1+db\wedge d\theta_2)$ integrates to zero and we see that 
\[
\int_Xc_1^2=\lim_{R\rightarrow \infty}\int_{\p B(R)}\frac{adb-bda}{r^2}.
\]
It follows from formula \ref{inversa_Dxi} that 
$$
a=-\frac{rg_b}{V}, \quad b=\frac{rg_a}{V}
$$
so that 
$$
adb-bda=\frac{r^2}{V^2}(-g_bdg_a+g_adg_b)-g_ag_bd\left(\frac{r}{V}\right)+g_ag_bd\left(\frac{r}{V}\right)
$$
hence 
$$
\frac{adb-bda}{r^2}=\frac{\Det(dg,g)}{V^2}
$$
and the claim follows from the fact that $V=\Det(g,f)$.
\end{proof}
We are now in a position to prove the main result in this section, namely:
\begin{prop}\label{c1}
Let $X$ be an unbounded toric manifold with a scalar-flat toric K\"ahler metric. Then
\[
\int_X c_1^2
\]
is finite.
\end{prop}
Before we begin with proof we set up some notation. 
\begin{nota}
We will write  
$$\oz, \, \oum,\, \oo,\, \ooo, \oooo$$ 
for functions on $\bbH$ which are bounded by 
$$C,\,\frac{C}{\rho},\,\frac{C}{\rho^2},\,\frac{C}{\rho^3}, \, \frac{C}{\rho^4}$$ 
respectively.  Here $C$ is some constant.
\end{nota}

\begin{proof}
Given equations \ref{f} and \ref{g} for the vectors $f$ and $g$ we can give an explicit formula for the $1$-form
$$
\frac{\Det(dg,g)}{\Det(g,f)^2}. 
$$
We have 
$$
dg=-\sum_{i=1}^{d-1}\frac{rdr+H_idH}{\rho_i^3}(\nu_{i+1}-\nu_i)
$$
so that $\Det(dg,g)$ is given by
$$
\frac{1}{2}\sum_{i=1}^{d-1}\frac{rdr+H_idH}{\rho_i^3}\Det(\nu,\nu_{i+1}-\nu_i)+\frac{1}{4}\sum_{i,j=1}^{d-1}\frac{rdr+H_idH}{\rho_i^3\rho_j}\Det(\nu_{j+1}-\nu_j,\nu_{i+1}-\nu_i)
$$
and from the formulas \ref{f} and \ref{g} we can write $\Det(g,f)$ as
$$
\begin{aligned}
&\Det(\nu,\nu_1)+\frac{1}{2}\sum_{i=1}^{d-1}\left(1-\frac{H_i}{\rho_i}\right)\Det(\nu,\nu_{i+1}-\nu_{i})-\sum_{i=1}^{d-1}\frac{1}{2\rho_i}\Det(\nu_1,\nu_{i+1}-\nu_{i})\\
&+\frac{1}{4}\sum_{i,j=1}^{d-1}\frac{1}{\rho_i}\left(1-\frac{H_j}{\rho_j}\right)\Det(\nu_{i+1}-\nu_i,\nu_{j+1}-\nu_j)\\
\end{aligned}
$$

Next we will analyze the behavior of $\Det(dg,g)/\Det(g,f)^2$ on $C_R$. We have that for any $i=1,\cdots d-1$
$$
\frac{1}{\rho}-\frac{1}{\rho_i}=\oo.
$$
We also have that for any $i,j=1,\cdots d-1$
$$
 \frac{H}{\rho^3}-\frac{H_i}{\rho_i^3}=\ooo, \, \frac{r}{\rho^3}-\frac{r}{\rho_i^3}=\ooo
$$
which implies that 
$$
\Det(dg,g)=\frac{(rdr+HdH)\Det(\nu,\nu_{d}-\nu_1)}{2\rho^3}+\ooo, \, \text{if}\quad \nu \ne 0
$$
and
$$
\frac{H}{\rho^4}-\frac{H_i}{\rho_i^3\rho_j}=\oooo, \, \frac{r}{\rho^4}-\frac{r}{\rho_i^3\rho_j}=\oooo
$$
together with the fact that
$$
\sum_{i,j=1}^{d-1}\Det(\nu_{i+1}-\nu_i,\nu_{j+1}-\nu_j) 
$$
this implies that
$$
\Det(dg,g)=\oooo,  \, \text{if} \quad \nu=0.
$$
On $C_R$, $\rho$ is constant so the $1$ form $rdr+HdH$ vanishes and we conclude that 
$$\Det(dg,g)=\ooo, \, \text{if}\quad \nu \ne 0$$ 
and 
$$\Det(dg,g)=\oooo, \, \text{if}\quad \nu= 0.$$ Also $\Det(g,f)$ is bounded from below in the case where $\nu\ne0$ and bounded from below by a constant $C$ times $1/\rho$ in the case where $\nu=0$. More precisely we have
$$
\Det(g,f)=\Det(\nu,\nu_1)+\frac{1}{2}\left(1-\frac{H}{\rho}\right)\Det(\nu,\nu_d-\nu_1)+\oum, \, \text{if}\quad \nu\ne 0
$$
and 
$$
\Det(g,f)=\frac{\Det(\nu_d,\nu_1)}{\rho}+\oo, \, \text{if}\quad \nu= 0.
$$
Hence we may conclude that in both cases $\Det(dg,g)/\Det(f,g)^2$  is $\oo$ on $C_R$ and therefore its integral on $C_R$ tends to zero as $R$ tends to infinity. 

In the region $L_R$, $\Det(dg,g)/\Det(f,g)^2$ is given by
\begin{equation}\label{c1LR}
\frac{\left(\frac{1}{2}\sum\frac{H_i}{|H_i|^3}\Det(\nu,\nu_{i+1}-\nu_i)+\frac{1}{4}\sum\frac{H_i}{|H_i|^3|H_j|}\Det(\nu_{j+1}-\nu_j,\nu_{i+1}-\nu_i)\right)dH
}{V^2}
\end{equation}
where 
$$
\begin{aligned}
V=&\Det(\nu,\nu_1)+\frac{1}{2}\sum_{i=1}^{d-1}\Det(\left(1-\frac{H_i}{|H_i|}\right)\nu-\frac{\nu_1}{2|H_i|},\nu_{i+1}-\nu_{i})\\
&+\frac{1}{4}\sum_{i,j=1}^{d-1}\frac{1}{|H_i|}\left(1-\frac{H_j}{|H_j|}\right)\Det(\nu_{i+1}-\nu_i,\nu_{j+1}-\nu_j).\\
\end{aligned}
$$
This is a rational form on each of the intervals $]-a_{i+1},-a_i[$  as well as on the intervals $]-\infty,-a_{d-1}[$ and $]-a_1,+\infty[$ which we can in principle integrate explicitly. Note that is has the appropriate behavior both at infinity and at the boundaries of the intervals so that it is integrable and therefore the corresponding integral on $L_R$ is finite. The lemma follows from this.
\end{proof}

\section{The Pontryagin class}\label{c_2}
The aim of this section is to show that the integral of the Pontryagin class of a scalar-flat K\"ahler toric metric on an unbounded toric $4$-manifold is finite. As for the corresponding calculation for $c_1^2$ which we carried out in the previous section, this involves some explicit knowledge about the asymptotic behavior of such metrics which were written down in \cite{as}.

We start by using the coordinates $(H,r)$ to give a more tractable formula for $\int_X p$. 
\begin{lemma}\label{lemmaT}
Let 
$$
M_H=\p_H \left(\frac{D\xi D\xi^t}{V}\right) V(D\xi D\xi^t)^{-1}-\frac{V_H}{V}\id
$$
 and 
 $$
 M_r=\p_r \left(\frac{D\xi D\xi^t}{V}\right) V(D\xi D\xi^t)^{-1}-\frac{V_r}{V}\id
 $$ 
respectively. 
Let
\begin{equation}\label{T}
\begin{aligned}
T=&\frac{\Tr (M_H^2) \Det(df,f)-\Tr (rM_rM_H) (\det(dg,f)+\Det(df,g))+\Tr (r^2M_r^2) \Det(dg,g)}{V^2}\\
 &+\frac{\Tr (rM_rdM_H-M_Hd(rM_r))}{V}.
\end{aligned}
\end{equation}
Then $\int_X p$ is given by the limit as $R$ tends to $\infty$ of the integral of $T$ on $\p B(R)$ i.e.
$$
\int_X p=\lim_{R\rightarrow \infty} \int_{\p B(R)} T
 $$
 \end{lemma}
\begin{proof}
By definition we have $p=\Tr (\cur\wedge \cur)$ where $\cur$ is the curvature of the connection induced on $TX$ by our scalar-flat toric K\"ahler metric. Now 
\[
\cur=d\theta+\theta\wedge \theta
\]
where the connection form $\theta$ is of type $(1,0)$. Because our manifold is of complex dimension $2$ we see that $\theta\wedge \theta\wedge \theta\wedge \theta$ is zero because it is of type $(4,0)$. As for $d\theta\wedge \theta\wedge\theta$, it is a sum of a form of type $(4,0)$ with a form of type $(3,1)$ both of which are zero. Hence 
\[
p=\Tr \left((d\theta+\theta\wedge \theta)\wedge (d\theta+\theta\wedge \theta)\right)=\Tr (d\theta \wedge d\theta)=d\Tr (\theta\wedge d\theta).
\]
By applying Stokes theorem we see that
\[
\int_{B(R)}p=\int_{\p B(R)} \Tr (\theta \wedge d\theta).
\]
We know that 
\[
\theta=\p \Hess u \Hess^{-1}u.
\]
Let $M_H$ and $M_r$ denote the matrices $\p_H \Hess u \Hess^{-1}u$ and $\p_r \Hess u \Hess^{-1}u$ respectively. Because 
$$
\Hess u =\frac{D\xi D\xi^t}{V}
$$
these matrices are also given by the formulas in the lemma. Namely 
$$
M_H=\p_H \left(\frac{D\xi D\xi^t}{V}\right) V(D\xi D\xi^t)^{-1}-\frac{V_H}{V}\id
$$
 and 
 $$
 M_r=\p_r \left(\frac{D\xi D\xi^t}{V}\right) V(D\xi D\xi^t)^{-1}-\frac{V_r}{V}\id.
 $$ 
Now
\[
\theta=M_H \p H +M_r \p r.
\]
We write this as
\[
\theta= \begin{pmatrix} M_H& M_r \end{pmatrix}\begin{pmatrix}\p H\\ \p r \end{pmatrix}.
\]
In lemma \ref{partial} we calculated $\p H$ and $\p r$. Replacing above we see that
\[
\theta= \begin{pmatrix} M_H& M_r \end{pmatrix}\begin{pmatrix}d H\\ d r \end{pmatrix}+i\begin{pmatrix} M_H& M_r \end{pmatrix}D\xi^{-1}\begin{pmatrix}d \theta_1\\ d \theta_2 \end{pmatrix}.
\]
Let $A$ and $B$ be matrices defined by
\[
\begin{pmatrix} A& B \end{pmatrix}=\begin{pmatrix} M_H& M_r \end{pmatrix}D\xi^{-1}.
\]
The matrices $A$ and $B$ are only function of $H$ and $r$, they do not depend on $\theta_1$ and $\theta_2$. We have
\[
\theta= M_HdH+ M_rdr +i(Ad \theta_1+B d \theta_2),
\]
and
\[
d\theta= dM_H\wedge dH+d M_r\wedge dr +i(dA\wedge d \theta_1+dB \wedge d \theta_2).
\]
By wedging these two together we see that 
\[
\theta \wedge d\theta= -(Ad \theta_1+B d \theta_2)\wedge(dA\wedge d \theta_1+dB \wedge d \theta_2)+M
\]
where $M$ is a matrix of two forms which integrate to zero on $\p B(R)$. In fact, the entries of the matrix $M$ are linear combinations of the forms $dH\wedge dr\wedge d \theta_1$ and $dH\wedge dr\wedge  d\theta_2$, which vanish on $\p B(R)$. From this we see that
\[
\int_X p=\lim_{R\rightarrow \infty }\int_{\p B(R)} \Tr (AdB-BdA)
\]
Now 
$$
\begin{pmatrix} A& B \end{pmatrix}=\begin{pmatrix} M_H& M_r \end{pmatrix}D\xi^{-1} 
$$
so that using formula \ref{inversa_Dxi} for the inverse of $D\xi$ it follows that 
$$
 A=\frac{1}{V}\left({f_bM_H}-{g_b}rM_r \right) 
$$
and 
$$
 B=-\frac{1}{V}\left({f_aM_H}-{g_a}rM_r \right). 
$$
By expanding out $\Tr (AdB-BdA)$ using the above expressions the result of the lemma follows in a straightforward manner.
\end{proof}
It is also possible to obtain expressions for the matrices $M_H$ and $M_r$ in terms of the functions $f$ and $g$. These expressions although slightly cumbersome will prove very important for our calculations. So we write them in a separate lemma.
\begin{lemma}
Let $M_H$ and $M_r$ be defined as in lemma \ref{lemmaT}. Then we have the following expression for $M_H$ 
\begin{equation}
\begin{aligned}\label{MH}
M_H=& \frac{\Det(g,f)g_H(f^\pp)^t+r^2\Det(g_H,g)g(g^\pp)^t}{V^2}\\
+&\frac{\Det(g_H,f)g(f^\pp)^t-\Det(g,f)f_H(g^\pp)^t}{V^2} \\
+& \frac{-\Det(g,f_H)f(g^\pp)^t+\frac{\Det(f_H,f)}{r^2}f(f^\pp)^t}{V^2}\\
-&\frac{\Det(g_H,f)+\Det(g,f_H)}{\Det (g,f)}\id,
\end{aligned}
\end{equation}
and a corresponding expression for $M_r$ namely
\begin{equation}\label{Mr}
\begin{aligned}
M_r=&\frac{ \Det(g,f)g_r(f^\pp)^t+r^2\Det(g_r,g)g(g^\pp)^t}{V^2}\\
&\frac{\Det(g_r,f)g(f^\pp)^t-\Det(g,f)f_r(g^\pp)^t}{V^2} \\
&+ \frac{-\Det(g,f_r)f(g^\pp)^t+\frac{\Det(f_r,f)}{r^2}f(f^\pp)^t}{V^2}\\
&+ \frac{2\Det(g,f)f(g^\pp)^t}{V^2r}\\
&-\frac{\Det(g_r,f)+\Det(g,f_r)}{\Det (g,f)}\id
\end{aligned}
\end{equation}
\end{lemma}
\begin{proof}
From formula \ref{Dxi} it is easy to see that 
 $$
 D\xi D\xi^t=gg^t+\frac{ff^t}{r^2}
 $$
 and 
 $$
 (D\xi D\xi^t)^{-1}=\frac{r^2}{V^2}\left(g^\pp(g^\pp)^t+\frac{f^\pp(f^\pp)^t}{r^2}\right)
 $$
where $f^\pp=(f_b,-f_a)$ and $g^\pp=(g_b,-g_a)$.  Also
$$
\p_H (D\xi D\xi^t)=g_Hg^t+gg_H^t+\frac{f_Hf^t}{r^2}+\frac{ff_H^t}{r^2}.
$$
By noting that for any vectors $u$ and $v$, 
$$
u^t v^{\pp}=u\cdot v^\pp =\Det(u,v)
$$ 
this gives an expression for $\p_H (D\xi D\xi^t)(D\xi D\xi^t)^{-1}$.

We also have that $V=r\Det D\xi=\det(g,f)$ and $V_H=\Det(g_H,f)+\Det(g,f_H)$. Now
$$
M_H=\p_H (D\xi D\xi^t)(D\xi D\xi^t)^{-1}-\frac{\p_H V}{V}\id.
$$
Substituting above we see that $M_H$ is given by 
\begin{equation}
\begin{aligned}
& \frac{\Det(g,f)g_H(f^\pp)^t+r^2\Det(g_H,g)g(g^\pp)^t}{V^2}\\
+&\frac{\Det(g_H,f)g(f^\pp)^t-\Det(g,f)f_H(g^\pp)^t}{V^2} \\
+& \frac{-\Det(g,f_H)f(g^\pp)^t+\frac{\Det(f_H,f)}{r^2}f(f^\pp)^t}{V^2}\\
-&\frac{\Det(g_H,f)+\Det(g,f_H)}{\Det (g,f)}\id
\end{aligned}
\end{equation}
as claimed. We can do a similar calculation for $M_r$.
\end{proof}
It is easy to see from the above formulas that $M_H$ is actually smooth up to the boundary of $\bbH$ except at the points $(-a_i, 0)$. The matrix-valued function $M_r$ is not. To deal with this fact we will often work with $rM_r$ which is smooth.

We start by working out an example. This is meant to help the reader follow the general calculations later on but also, it will turn out to be important in the proof of our main result.
\begin{exa}\label{example2edges}
Let $X(\nu_1,\nu_d)$ be the toric orbifold whose polytope $P_0$ has exactly two edges which are unbounded edges with normals $\nu_1$ and $\nu_d$. Endow $X(\nu_1,\nu_d)$ with the ALE metric corresponding to the parameter $\nu=0$ from \cite{as}.  Then  we will show that
$$
\int_{X(\nu_1,\nu_d)}p=0,
$$
what is more
$$
\lim_{R \rightarrow +\infty}\int_{L_R}T=0, \, \lim_{R \rightarrow +\infty}\int_{C_R}T=0.
$$

Let $f$, $h$, $M_H$ and $M_r$ be defined by the expressions \ref {f}, \ref{g}, \ref{MH} and \ref{Mr} respectively. The formulas for $f$ and $g$ simplify in this context to give 
$$
f=\nu_1+\frac{1}{2}\left(1-\frac{H}{\rho}\right)(\nu_{d}-\nu_{1})
$$
and 
$$
g=\frac{(\nu_{d}-\nu_{1})}{2\rho}.
$$
 We have 
$$
df=\frac{r(-rdH+Hdr)}{2\rho^3}(\nu_{d}-\nu_{1})=\frac{r(-rdH+Hdr)}{\rho^2}g,
$$ 
and
$$
dg=-\frac{HdH+rdr}{2\rho^3}(\nu_{d}-\nu_{1})=-\frac{HdH+rdr}{\rho^2}g.
$$
Or to write things in another way
$$
f_H=-\frac{r^2}{\rho^2}g, \, f_r=\frac{rH}{\rho^2}g, \, g_H=-\frac{H}{\rho^2}g, \, g_r=-\frac{r}{\rho^2}g.
$$
We have $$\Det(df,g),\Det(f_H,g),\Det(f_r,g)=0,$$ and also $$\Det(dg,g),\Det(g_H,g),\Det(g_r,g)=0.$$ Further
$$
V=\frac{\Det(\nu_d,\nu_1)}{2\rho}.
$$
Formula \ref{T} for $T$ simplifies to yield
\begin{equation}
\begin{aligned}
T=&\frac{\Tr (M_H^2) \Det(df,f)-\Tr (rM_rM_H) \det(dg,f)}{V^2}\\&
+\frac{\Tr (rM_rdM_H-M_Hd(rM_r))}{V}.
\end{aligned}
\end{equation}
Now 
$$
\Det(df,f)=\frac{r(-rdH+Hdr)}{\rho^2}V,
$$
and
$$
\Det(dg,f)=-\frac{HdH+rdr}{\rho^2}V
$$
so that we can substitute in the above formula for $T$ to get 
\begin{equation}\label{T0}
\begin{aligned}
T=&\frac{\Tr (M_H^2) r(-rdH+Hdr)+\Tr (rM_rM_H)(HdH+rdr) }{\rho^2V}\\
&\frac{d\Tr(M_HrM_r)}{V}-\frac{2\Tr (M_Hd(rM_r))}{V}
\end{aligned}
\end{equation}
or
\begin{equation}
\begin{aligned}
T=&\frac{\Tr (M_H^2) r(-rdH+Hdr)+\Tr (rM_rM_H)(HdH+rdr) }{\rho^2V}\\
&\frac{2\Tr(d(M_H)rM_r}{V}-\frac{\Tr d(M_HrM_r))}{V}.
\end{aligned}
\end{equation}
We have 
$$
g_H(f^\pp)^t=\frac{-Hg (f^\pp)^t}{\rho^2},\, \, f_H(g^\pp)^t=\frac{-r^2g (g^\pp)^t}{\rho^2}, \, ,
$$
and
$$
\Det(g_H,f)=\frac{-HV}{\rho^2},\,\Det(f_H,f)=\frac{-r^2V}{\rho^2}
$$
The expression in equation \ref{MH}  simplifies accordingly to give
$$
\begin{aligned}
M_H=& \frac{-2Hg (f^\pp)^t+r^2g (g^\pp)^t-f(f^\pp)^t}{\rho^2V} +\frac{H}{\rho^2}\id,
\end{aligned}
$$
Similarly we have
$$
g_r(f^\pp)^t=\frac{-rg (f^\pp)^t}{\rho^2},\, \, f_r(g^\pp)^t=\frac{rHg (g^\pp)^t}{\rho^2}, \, ,
$$
and
$$
\Det(g_r,f)=\frac{-rV}{\rho^2},\,\Det(f_r,f)=\frac{rHV}{\rho^2}
$$
replacing in equation \ref{Mr} the expression for $M_r$ becomes
$$
\begin{aligned}
M_r=&-\frac{ 2rg (f^\pp)^t+rHg (g^\pp)^t}{\rho^2V}+ \frac{Hf(f^\pp)^t}{r\rho^2V}+ \frac{2f(g^\pp)^t}{Vr}+\frac{r}{\rho^2}\id
\end{aligned}
$$ 
so that 
$$
rM_r=-\frac{ 2r^2g (f^\pp)^t+r^2Hg (g^\pp)^t-Hf(f^\pp)^t}{\rho^2V}+  \frac{2f(g^\pp)^t}{V}+\frac{r^2}{\rho^2}\id
$$
There are two separate calculations to do in order to calculate $\int_{\p B(R)} T$ namely $\int_{C_R} T$ and $\int_{L_R} T$. We start with the later.

Assume that $r=0$. Replacing in equation \ref{T0} we see that
\begin{equation}\label{T0r=0}
T=\frac{\Tr (rM_rM_H)HdH}{\rho^2V}+\frac{d\Tr(M_HrM_r)}{V}-\frac{2\Tr (M_Hd(rM_r))}{V}.
\end{equation}
We can also simplify the expressions for $M_H$ and $M_r$ further so that we have
$$
M_H=- \frac{2Hg (f^\pp)^t+f(f^\pp)^t}{\rho^2V} +\frac{H}{\rho^2}\id,
$$
and
$$
rM_r=\frac{ Hf(f^\pp)^t}{\rho^2V}+  \frac{2f(g^\pp)^t}{V}.
$$
It follows that 
$$
\Tr(M_HrM_r)=-\frac{2H}{\rho^2}
$$
and 
$$
d(rM_r)=d\left(\frac{ H}{\rho^2V}\right)f(f^\pp)^t+  2d\left(\frac{1}{V}\right)f(g^\pp)^t-\frac{2Hf (g^\pp)^tdH}{\rho^2V}.
$$
But on $L_R$,
$$
d\left(\frac{1}{V}\right)=-\frac{V_HdH}{V^2}=-\frac{\Det(g_H,f)dH}{V^2}=\frac{HVdH}{\rho^2V^2}=\frac{HdH}{\rho^2V}
$$
so that 
$$
 2d\left(\frac{1}{V}\right)f(g^\pp)^t-\frac{2Hf (g^\pp)^tdH}{\rho^2V}=0.
$$
and 
$$
d(rM_r)=d\left(\frac{ H}{\rho^2V}\right)f(f^\pp)^t.
$$
Therefore we have that 
$$
M_Hd(rM_r)=\frac{H}{\rho^2}d\left(\frac{ H}{\rho^2V}\right)f(f^\pp)^t
$$
and $M_Hd(rM_r)$ has zero trace. We can now return to equation \ref{T0r=0}. We have 
$$
T=-\frac{2H^2dH}{\rho^4V}-\frac{2}{V}d\left(\frac{H}{\rho^2}\right)
$$
i.e.
$$
T=-\frac{2H^2dH}{\rho^4V}-\frac{2dH}{V\rho^2}+\frac{4H^2dH}{V\rho^4}
$$
which is zero when $r=0$. Hence $\int_{L_R}T$ is actually zero in this example.

Let us now turn our attention to $\int_{C_R}T$. Note that on $C_R$ we have $HdH+rdr=0$ and $dg=0$. Substituting in equation \ref{T0} we see that 
\begin{equation}\label{T0R}
\begin{aligned}
T=&\frac{\Tr (M_H^2) r(-rdH+Hdr)}{\rho^2V}+\frac{2\Tr(d(M_H)rM_r)}{V}-\frac{\Tr d(M_HrM_r)}{V}.
\end{aligned}
\end{equation}
From the expression
$$
M_H= \frac{-2Hg (f^\pp)^t+r^2g (g^\pp)^t-f(f^\pp)^t}{\rho^2V} +\frac{H}{\rho^2}\id,
$$
we se that
$$
\Tr M_H^2=\frac{2(\rho^2+r^2)}{\rho^4}=\frac{2}{\rho^2}
$$
and also that 
$$
dM_H= -\frac{2dHg (f^\pp)^t+2Hg(df^\pp)^t-2rdrg (g^\pp)^t+df(f^\pp)^t+f(df^\pp)^t}{\rho^2V} +\frac{dH}{\rho^2}\id,
$$
where we have used the fact on $C_R$, $\rho$ and therefore $V$ are constants. By using the fact that 
$$
df=\frac{r}{\rho^2}(-rdH+Hdr)g
$$
it follows that 
$$
\begin{aligned}
dM_H&= -\left(2\rho^2dH+{r}(-rdH+Hdr)\right)\frac{g(f^\pp)^t}{\rho^4V}-{r}(-rdH+Hdr)\frac{f(g^\pp)^t}{\rho^4V} \\
&-\left({2Hr}(-rdH+Hdr)-2r\rho^2dr \right)\frac{g(g^\pp)^t}{\rho^4V}+\frac{dH}{\rho^2}\id,
\end{aligned}
$$
We also have that 
$$
rM_r=-\frac{ 2r^2g (f^\pp)^t+r^2Hg (g^\pp)^t-Hf(f^\pp)^t}{\rho^2V}+  \frac{2f(g^\pp)^t}{V}+\frac{r^2}{\rho^2}\id
$$
and by multiplying out we conclude that 
$$
\Tr(rM_rdM_H)=-\frac{2Hd\rho}{\rho^3}
$$
and that 
$$
\Tr(rM_rM_H)=-\frac{2H}{\rho^2}
$$
so that 
$$
T=\frac{2}{\rho^2V}\left(\frac{ r(-rdH+Hdr)}{\rho^2V}+dH\right)
$$
which is zero on $C_R$. So, for the ALE metric on $X(\nu_1,\nu_d)$ we have 
$$
\int_{X(\nu_1,\nu_d)}p=0.
$$
\end{exa}

We are now in a position to prove the main result of this section
\begin{prop}\label{finitep}
Let $X$ be a strictly unbounded toric surface with a scalar-flat toric K\"ahler metric from \cite{as} with parameter $\nu$. Then
\[
\int_X p
\]
is finite. In fact if we set $a_0=-\infty$ and $a_d=+\infty$, 
$$
\int_X p= \sum_{i=1}^d\left[\frac{2V_H}{V^2}\right]_{-a_i}^{-a_{i-1}}+4\int_{r=0}\frac{\Det(g_H,g)dH}{V^2}.
 $$
Here $V=\Det(g,f)$ where $g$ and $f$ are given by the formulas \ref{g} and \ref{f}.  
\end{prop}
\begin{proof}
To do this it is enough to show that
$$
\lim_{R\rightarrow \infty} \int_{\p B(R)} T
$$
is finite. As in the example there are two terms in the above integral which we treat separately. First we will study 
$$
\int_{C_R} T .
$$
 To this end we study the asymptotic behavior of $M_H$ and $M_r$. 
 
 Consider first the case when $\nu\ne 0$. It follows from the formulas \ref{f} and \ref{g} that $f$  and $g$ are bounded, $df$ is $\oum$ and $dg$ is $\oo$. When $\nu\ne 0$ this gives us the asymptotic behavior of $M_H$. Namely
$$
M_H=O_H(0)+O_H(-1)+O_H(-2)
$$
where $O_H(0),O_H(-1)$ and $O_H(-2)$ are $\oz$, $\oum$ and $\oo$ respectively. They are given explicitly by 
$$
O_H(0)=\frac{r^2\Det(g_H,g)}{V^2}g(g^\pp)^t;
$$
$$
O_H(-1)=-\frac{1}{V^2}\left( \Det(g,f)f_H(g^\pp)^t+\Det(g,f_H)f(g^\pp)^t \right)-\frac{\Det(g,f_H)}{\Det (g,f)}\id;
$$
$$
\begin{aligned}
O_H(-2)=\frac{1}{V^2}&\left(\Det(g,f)g_H(f^\pp)^t+\Det(g_H,f)g(f^\pp)^t+\frac{\Det(f_H,f)}{r^2}f(f^\pp)^t \right)\\
&-\frac{\Det(g_H,f)}{\Det (g,f)}\id.
\end{aligned}
$$
respectively. From the formula  \ref{Mr} we write 
$$
M_r=O_r(0)+O_r(-1)+O_r(-2)
$$
where $O_r(0),O_r(-1)$ and $O_r(-2)$ are $\oz$, $\oum$ and $\oo$ respectively. They are given explicitly by 
$$
O_r(0)=\frac{r^2\Det(g_r,g)}{V^2}g(g^\pp)^t;
$$
$$
\begin{aligned}
O_r(-1)=-\frac{1}{V^2}&\left( \Det(g,f)f_r(g^\pp)^t+\Det(g,f_r)f(g^\pp)^t-\frac{2\Det(g,f)}{r}f(g^\pp)^t\right)\\
&-\frac{\Det(g,f_r)}{\Det (g,f)}\id;
\end{aligned}
$$
$$
\begin{aligned}
O_r(-2)=\frac{1}{V^2}&\left(\Det(g,f)g_r(f^\pp)^t+  \Det(g_r,f)g(f^\pp)^t+\frac{\Det(f_r,f)}{r^2}f(f^\pp)^t\right)\\
&-\frac{\Det(g_r,f)}{\Det (g,f)}\id
\end{aligned}
$$
respectively. From these calculation it follows that $\Tr (M_H^2)$ is bounded. But in fact: 
\begin{itemize}
\item its highest order term is the trace of $O_H(0)^2$ which is zero because $$g(g^\pp)^tg(g^\pp)^t=0,$$
\item as for the term in $\oum$ of this trace it is given by the trace of $$2O_H(0)O_H(-1)$$ which is zero as well because the matrices 
$$g(g^\pp)^tf_H(g^\pp)^t,\, g(g^\pp)^tf(g^\pp)^t,\,  g(g^\pp)$$ 
all have zero trace. In fact they are all multiples of $g(g^\pp)$. 
\end{itemize}
From these two observations we conclude that $\Tr (M_H)^2$ is in fact $\oo$. Similarly $\Tr M_r^2=\oo$ and $\Tr (M_HM_r)=\oo$. It follows that
\begin{itemize}
\item $\Tr(M_H^2)\Det(df,f)$ is $\ooo$,
\item $\Tr (r^2 M_r^2)\Det(dg,g)$ is $\oo$
\item $\Tr (rM_rM_H)(\det(dg,f)+\det (df,g))$ is $\oo$.
\end{itemize}
We can conclude that the integral of the sum of the above expressions on $C_R$ tends to zero as $R$ tends to $\infty$. We are left with studying 
$$
\int_{C_R}\frac{\Tr(M_Hd(rM_r)-rM_rdM_H)}{V}.
$$
Let us start by considering $\Tr (M_rdM_H)$. This is in principle $\oum$. But again 
\begin{itemize}
\item its highest order term which is $O_r(0)dO_H(0)$ has zero trace as it is given as a linear combination of the traces of the matrices $g(g^\pp)^tg(g^\pp)^t$, $g(g^\pp)^tdg(g^\pp)^t$ or $g(g^\pp)^tg(dg^\pp)^t$ all of which vanish because $g(g^\pp)^t$ has zero trace or $(g^\pp)^tg=0$,
\item the next term, namely $O_r(0)dO_H(-1)+O_r(-1)dO_H(0)$ is a priori of order $\oo$, but in fact by carrying out explicit calculations one sees that this term reduces to 
$$
\begin{aligned}
&\frac{2r^2\Det(g,dg)}{V^3} \left( \Det(g_r,g)\Det(f_H,g)+\Det(f_r,g)\Det(g_H,g)\right)\\
&-\frac{2r^2\Det(g,dg)\Det(g_H,g)}{V^2r}
\end{aligned}
$$
which is $\ooo$.
\end{itemize}
From this we conclude that $\Tr (rM_rdM_H)$ is $\oo$. Now as we have seen 
$$
\Tr (M_HM_r)=\oo,
$$ 
which implies that
$$
\Tr (M_HrM_r)=\oum, \, \Tr (d M_HrM_r)=\oo
$$
therefore 
$$
\Tr (d(rM_r)M_H)=d\Tr (M_HrM_r)-\Tr (rM_rdM_H)=\oo.
$$
Because $V$ is bounded from below it follows that 
$$
 \int_{ C_R} T
 $$
 has limit zero as $R$ tends to $+\infty$. 
 
 Next we consider the case when $\nu=0$. From formulas \ref{f} and \ref{g} we see that in this case $f$ is bounded but $g=\oum$. On the other hand $V$ is $\oum$. Formulas \ref{MH} and \ref{Mr} still hold but now they imply a different asymptotic behavior for $M_H$ and $M_r$. Namely each of the terms in the sum is now $\oum$. Therefore $M_H=\oum$ and $rM_r$ is bounded. It follows that
 \begin{itemize}
\item $\Tr(M_H^2)\Det(df,f)$ is $\ooo$,
\item $\Tr (r^2 M_r^2)\Det(dg,g)$ is $\ooo$,
\item $\Tr (rM_rM_H)(\det(dg,f)+\det (df,g))$ is $\ooo$,
\item $\Tr (rM_rdM_H-M_Hd(rM_r))\Det(g,f)$ is $\ooo$.
\end{itemize}
Each of the terms will be multiplied by $1/V^2$ which is of order $\rho^2$ so that we see that the integral over $C_r$ is actually $\oum$ and therefore in this case there might be a contribution towards $\int_X p$. 

Consider the polytope $P_0$ with only two unbounded edges with normals $\nu_1$ and $\nu_d$. This is the polytope of the toric orbifold $X(\nu_1,\nu_d)$ from example \ref{example2edges} for which the construction of \cite{as} also works. We can associate to it functions $f_0$ and $g_0$ via the formulas
$$
f_0=\nu_1+\frac{1}{2}\left(1-\frac{H}{\rho}\right)(\nu_{d}-\nu_{1})
$$
and 
$$
g_0=\frac{(\nu_{d}-\nu_{1})}{2\rho}.
$$
These are as in \ref{f} and \ref{g} respectively for the case when there are only two normals and $\nu=0$. Let $M_H$ and $M_r$ be defined by $f_0$ and $g_0$ as in equations \ref{MH} and \ref{Mr} respectively and $T_0$ be given by equation \ref{T} accordingly. The point is that 
$$
f-f_0=\oum, \, g-g_0=\oo
$$
and this implies that $T$ differs from $T_0$ by a term in $\oo$. Thus 
$$
\lim_{R \rightarrow 0} \int_{C_R} T-T_0=0
$$
We know from example \ref{example2edges} that 
$$
\lim_{R \rightarrow 0} \int_{C_R} T_0=0,
$$ 
therefore
$$
\lim_{R \rightarrow 0} \int_{C_R} T=0.
$$

We will now analyze the term 
 $$
\int_{L_R} T .
$$
When $r=0$ we have that $f$ is constant since 
$$
\left(1-\frac{H_i}{\rho_i}\right)
$$
is locally constant for all $i=1,\dots d-1$. It is either equal to $0$ or $2$ so that $f_r,f_H,df$ are all zero on $L_R$. Now consider equation \ref{MH} for $M_H$. This simplifies slightly when $r=0$ to give
 \begin{equation}\label{MH0}
M_H= \frac{g_H(f^\pp)^t}{V}+\frac{\Det(g_H,f)g(f^\pp)^t}{V^2}+ \frac{{\Det(f_H,f)}f(f^\pp)^t}{r^2V^2}-\frac{\Det(g_H,f)}{V}\id\\
\end{equation} 
We can also see that $g_r=0$ on $L_R$ and equation \ref{Mr} for $M_r$ also simplifies. We get 
$$
rM_r= \frac{{\Det(f_r,f)}f(f^\pp)^t}{rV^2}+ \frac{2f(g^\pp)^t}{V}.
$$
But we can easily see from formulas \ref{f} and \ref{g}  that
$$
\frac{f_r}{r}=\frac{1}{2}\sum_{i=1}^{d-1} \frac{H_i}{\rho_i^3}(\nu_{i+1}-\nu_i)=-g_H
$$ 
so the above becomes
 \begin{equation}\label{Mr0}
rM_r=- \frac{{\Det(g_H,f)}f(f^\pp)^t}{V^2}+ \frac{2f(g^\pp)^t}{V}.
\end{equation}
Because $(f^\pp)^tf=0$, it follows from these two formulas that
$$
M_HrM_r=\frac{\Det(g_H,f)}{V}\left( \frac{{\Det(g_H,f)}f(f^\pp)^t}{V^2}- \frac{2f(g^\pp)^t}{V}\right)
$$
and thus 
\begin{equation}\label{trMHMr}
\Tr(M_HrM_r)=\frac{2\Det(g_H,f)}{V},
\end{equation}
implying that on $L_R$
$$
d\Tr(M_HrM_r)=2\left(\frac{\Det(g_{HH},f)}{V}-\frac{\Det(g_{H},f)^2}{V^2}\right).
$$
One can also check using equation \ref{Mr0} that $\Tr (r^2M_r^2)=4$. Equation \ref{T} for $T$ in turn also simplifies and we see that $T$ is given by
$$
\begin{aligned}
&\frac{-\Tr (rM_rM_H) \det(dg,f)+\Tr (r^2M_r^2) \Det(dg,g)}{V^2}\\
&+\frac{d\Tr (rM_rM_H)}{V}-\frac{2\Tr (M_Hd(rM_r))}{V}
\end{aligned}
 $$
hence on $L_R$ we have
$$
\begin{aligned}
T=\frac{-2 \det(g_H,f)^2dH}{V^3}+\frac{4\Det(g_H,g)dH}{V^2}+\frac{d\Tr (rM_rM_H)}{V}-\frac{2\Tr (M_Hd(rM_r))}{V}.
\end{aligned}
 $$
Next we calculate $\Tr (M_Hd(rM_r))$. It follows from equation \ref{Mr0} that on $L_R$
$$
\frac{d(rM_r)}{dH}=-\frac{\p}{\p H}\left( \frac{{\Det(g_H,f)}}{V^2}\right) f(f^\pp)^t+ \frac{2f(g_H^\pp)^t}{V}-\frac{2\Det(g_H,f)f(g^\pp)^t}{V^2}.
$$
 and $\Tr(M_Hd(rM_r))$ is given by
 $$
 -\frac{\Det(g_H,f)}{V}\left(  \frac{2\Det(f,g_H)}{V}-\frac{2\Det(g_H,f)\Det(f,g)}{V^2}\right)=0
 $$
so that 
$$
\begin{aligned}
T=\frac{2dH}{V^2}\left(\frac{-2 \det(g_H,f)^2}{V}+{2\Det(g_H,g)}{}+{\Det(g_{HH},f)}{}\right).
\end{aligned}
 $$
 or 
 $$
\begin{aligned}
T=\frac{4\Det(g_H,g)dH}{V^2}+2\left(\frac{V_{HH}}{V^2}-\frac{2 V_H^2}{V^3}\right)dH.
\end{aligned}
 $$
Integrating 
$$
\frac{-2 V_H^2}{V^3}
$$
by parts we see that  
$$
\begin{aligned}
\int_a^b T_{|L_R}=\left[\frac{2V_H}{V^2}\right]_a^b+4\int_a^b\frac{\Det(g_H,g)dH}{V^2}.
\end{aligned}
 $$
This proves the formula in the statement of the proposition. To show that this formula implies finiteness of $\int_X p$ note that 
$$
\begin{aligned}
V=&\det(\nu,\nu_1)+\sum_{i=1}^{d-1}\frac{\Det(\nu_{i+1}-\nu_i,\nu_1)}{2\rho_i}+\sum_{i=1}^{d-1}\frac{1}{2}\left(1-\frac{H_i}{\rho_i}\right)\Det(\nu,\nu_{i+1}-\nu_i)\\
&+\sum_{i,j=1}^{d-1}\frac{1}{4\rho_i}\left(1-\frac{H_j}{\rho_j}
\right)\Det(\nu_{i+1}-\nu_i,\nu_{j+1}-\nu_j).
\end{aligned}
$$
Since $V_H=\Det(g_H,f)$ on $L_R$
$$
\begin{aligned}
V_H=&-\sum_{i=1}^{d-1}\frac{H_i\Det(\nu_{i+1}-\nu_i,\nu_1)}{2\rho_i^3}\\
&-\sum_{i,j=1}^{d-1}\frac{H_i}{4\rho_i^3}\left(1-\frac{H_j}{\rho_j}
\right)\Det(\nu_{i+1}-\nu_i,\nu_{j+1}-\nu_j).
\end{aligned}
$$
We can also write down a formula for $\Det(g_H,g)$. 
$$
\begin{aligned}
\Det(g_H,g)=&\sum_{i=1}^{d-1}\frac{H_i\Det(\nu,\nu_{i+1}-\nu_i)}{2\rho_i^3}\\
-&\sum_{i,j=1}^{d-1}\frac{H_i}{4\rho_i^3\rho_j}\Det(\nu_{i+1}-\nu_i,\nu_{j+1}-\nu_j).
\end{aligned}
$$
When $H$ is close to $-a_i$ we have that
$$
V= \frac{C_1}{|H_i|}+O(0),\quad V_H=\frac{C_2H_i}{|H_i|^3}+O(0), \quad \Det(g_H,g)=\frac{C_3H_i}{|H_i|^3}+O(0)
$$
where $O(0)$ denotes a bounded function of $H$ in a neighborhood of $-a_i$ and $C_1,C_2$ and $C_3$ are constants. This then implies that 
$$
\frac{\Det(g_H,g)}{V^2}, \frac{V_H}{V^2}
$$
are bounded which in turn implies that $T$ is bounded close to $-a_i$. Similarly when $H$ tends to infinity it is easy to see that $T$ is $\oh$ (i.e. $T$ is bounded by $C/H$ for some constant $C$ on $r=0$) both when $\nu\ne 0$ and $\nu=0$. We can thus conclude that 
$$
 \int_{L_R} T
 $$
is bounded. This concludes the proof. 
 \end{proof}
Before we end this section we sum up our results in the following proposition
\begin{prop}
Let $X$ be a strictly unbounded toric surface with a scalar-flat toric K\"ahler metric from \cite{as} with parameter $\nu$. Then
$$
\begin{aligned}
\int_X |\TR|^2= \sum_{i=1}^d\left[\frac{2V_H}{V^2}\right]_{-a_i}^{-a_{i-1}}+3\int_{\bbR}\frac{G}{V^2}
\end{aligned}
$$
where $a_0=-\infty$ and $a_d=+\infty$, with
$$
\begin{aligned}
V=&\Det(\nu,\nu_1)+\frac{1}{2}\sum\Det(\left(1-\frac{H_i}{|H_i|}\right)\nu-\frac{\nu_1}{2|H_i|},\nu_{i+1}-\nu_{i})\\
&+\frac{1}{4}\frac{1}{|H_i|}\left(1-\frac{H_j}{|H_j|}\right)\Det(\nu_{i+1}-\nu_i,\nu_{j+1}-\nu_j),
\end{aligned}
$$
$$
V_H=\frac{1}{2}\sum\frac{H_i}{|H_i|^3}\Det(\nu_1,\nu_{i+1}-\nu_i)+\frac{1}{4}\sum\frac{H_j}{|H_j|^3}\left(1-\frac{H_i}{|H_i|}\right)\Det(\nu_{i+1}-\nu_i,\nu_{j+1}-\nu_j)
$$
and
$$
G=\frac{1}{2}\sum\frac{H_i}{|H_i|^3}\Det(\nu,\nu_{i+1}-\nu_i)+\frac{1}{4}\sum\frac{H_i}{|H_i|^3|H_j|}\Det(\nu_{j+1}-\nu_j,\nu_{i+1}-\nu_i)
$$
\end{prop}
\section{Explicit calculations}\label{example}

In this sections we use the above results to calculate $\int_X|\TR|^2$ for an important case namely that of minimal resolutions of cyclic singularities endowed with scalar-flat metrics. In \cite{as} this case is also treated in more detail.

Gravitational instantons are complete hyperk\"ahler metrics on non-compact manifolds. These are always K\"ahler and Ricci-flat. Some of the metrics constructed in \cite{as} turn out to be Ricci-flat (but note that some are not). These corresponds to special choices of the parameter $\nu$ indexing the metrics. 

In a recent preprint Atiyah and LeBrun derive a beautiful Gauss-Bonnet formula for metrics with edge cone singularities. As an application of their formula they give an explicit expression for the Calabi functional of gravitational instantons on spaces which are of the form $\bbC^2/\Gamma$, where $\Gamma$ is a finite group of $SU(2)$, in terms of $\Gamma$. The authors only consider Ricci-flat metrics whereas our calculation applies to a larger family of scalar-flat metrics. On the other hand they do not limit themselves to toric metrics as we do here. Also the methods appearing in \cite{al} are likely to have other applications.

Let $\Gamma=\Gamma_d$ is the finite subgroup of $SU(2)$, of order
$d-1\in\mathbb{N}$, generated by
\begin{displaymath}
\begin{pmatrix}
  e^{\frac{2i\pi}{d-1}}&0\\
 0& e^{\frac{2i\pi(d-2)}{d-1}}
\end{pmatrix}.
\end{displaymath}
Let $X$ be the minimal toric resolution of $\mathbb{C}^2/\Gamma$. Its moment polygon is
$SL(2,\bbZ)$ equivalent to one with normals $\nu_1=(0,1)$, $\nu_2=(1,0)$, \dots,
$\nu_{d}=(d-1,-(d-2))$. Endow $X$ with a scalar-flat K\"ahler toric metric with parameter $\nu$ as constructed in \cite{as}. Such a metric exits and is complete as long as $\Det(\nu,\nu_1),\Det(\nu,\nu_d)>0$. This implies that 
$$
\alpha>0, \quad \alpha-(\alpha+\beta)(d-1)>0.
$$

We start by explicitly calculating $\int_X c_1^2$ using formula  \ref{c1LR}  in this case. Namely we prove the following lemma
\begin{lemma}
Let $X$ be the minimal toric resolution of a cyclic singularity of order $d-1$. Consider a scalar-flat metric from \cite{as} corresponding to an admissible parameter $\nu=(\alpha,\beta)$ i.e. 
$$
\alpha>0, \quad \alpha-(\alpha+\beta)(d-1)>0.
$$
Then
$$
\int_X c_1^2=(\alpha+\beta)\left(\frac{1}{\alpha}-\frac{1}{\alpha-(\alpha+\beta)(d-1)}\right)
$$
\end{lemma}
\begin{proof}
 As we have mentioned before, $X$ is the minimal toric resolution of $\mathbb{C}^2/\Gamma$. Its moment polygon is
$SL(2,\bbZ)$ equivalent to one with normals $\nu_1=(0,1)$, $\nu_2=(1,0)$, \dots,
$\nu_{d}=(d-1,-(d-2))$. Endow $X$ with one of the scalar-flat metrics from \cite{as}. For all $i=1,\cdots d-1$, $\nu_{i+1}-\nu_i=(1,-1)$ and the form $\Det(dg,g)/\Det(f,g)^2$ simplifies a great deal as a consequence. More precisely it follows from the expression \ref{c1LR} for 
$$
\frac{\Det(dg,g)}{\Det(f,g)^2}
$$
that we have 
$$
\frac{\Det(dg,g)}{\Det(f,g)^2}=-\frac{\frac{\alpha+\beta}{2}\sum_{i=1}^{d-1}\frac{\sg{H_i}}{H_i^2} dH
}{
\left( \alpha-\frac{\alpha+\beta}{2}\sum_{i=1}^{d-1}(1-\frac{H_i}{|H_i|})+\sum_{i=1}^{d-1}\frac{\sg{H_i}}{2H_i} \right)^2},
$$
where $\sg{H_i}$ is $1$ if $H_i\geq 0$ and $-1$ otherwise. Now 
$$
\sum_{i=1}^{d-1}\left(1-\frac{H_i}{|H_i|}\right)
$$ 
is constant on the intervals $]-a_{i+1},-a_i[$  as well as on the intervals $]-\infty,-a_{d-1}[$ and $]-a_1,+\infty[$ so that 
$$
-\sum\frac{\sg{H_i}}{2H_i^2}
$$
is the derivative of 
$$
V=\alpha-\frac{\alpha+\beta}{2}\sum_{i=1}^{d-1}\left(1-\frac{H_i}{|H_i|}\right)+\sum_{i=1}^{d-1}\frac{\sg{H_i}}{2H_i} 
$$
and therefore 
$$
\frac{\frac{\alpha+\beta}{2}\sum_{i=1}^{d-1}\frac{\sg{H_i}}{H_i^2} 
}{
\left( \alpha-\frac{\alpha+\beta}{2}\sum_{i=1}^{d-1}(1-\frac{H_i}{|H_i|})+\sum_{i=1}^{d-1}\frac{\sg{H_i}}{2H_i} \right)^2}=-(\alpha+\beta)\frac{V_H}{V^2}. 
$$
As we have seen before
$$
\int_X c_1^2=\lim_{R\rightarrow \infty}\int_{C_R}\frac{\Det(dg,g)}{\Det(g,f)^2}+\lim_{R\rightarrow \infty}\int_{L_R}\frac{\Det(dg,g)}{\Det(g,f)^2}
$$
and the first limit in the sum is zero. As for the second limit it is 
$$
\int_\bbR (\alpha+\beta)\frac{V_H}{V^2}=-(\alpha+\beta)\sum_{i=1}^d\left[\frac{1}{V}\right]_{-a_i}^{-a_{i-1}},
$$
where we set $a_0=-\infty$ and $a_d=+\infty$. We have 
$$
V=\alpha-\frac{\alpha+\beta}{2}\sum_{i=1}^{d-1}\left(1-\frac{H_i}{|H_i|}\right)+\sum_{i=1}^{d-1}\frac{\sg{H_i}}{2H_i}
$$
and near $-a_i$ 
$$
V=\frac{\sg{H_i}}{2H_i}+O(0)
$$
with $O(0)$ being a bounded function near $-a_i$. This shows that $1/V$ vanishes at  $-a_i$ for $i=1,\cdots d-1$. The limit of $V$ at $+\infty$ is $\alpha$ because
$$
H>>0 \implies 1-\frac{H_i}{|H_i|}=0, \,\forall i=1,\cdots d-1.
$$
Similarly
$$
H<<0 \implies 1-\frac{H_i}{|H_i|}=2, \,\forall i=1,\cdots d-1.
$$
so that the limit of $V$ at $-\infty$ is $\alpha-(\alpha+\beta)(d-1)$ and
$$
\int_X c_1^2=(\alpha+\beta)\left(\frac{1}{\alpha-(\alpha+\beta)(d-1)}-\frac{1}{\alpha}\right).
$$
Notice that when $\nu=0$ 
$$
g=\sum_{i=1}^{d-1}\frac{v}{\rho_i}
$$ 
where we set $v=(1,-1)$ and thus
$$
dg=\sum_{i=1}^{d-1}d\left(\frac{1}{\rho_i}\right)v
$$
so that $\Det(dg,g)=0$ and $\int_X c_1^2=0$.
\end{proof}
We can use similar calculations to prove the following lemma
\begin{lemma}\label{Ad}
Let $X$ be the minimal toric resolution of a cyclic singularity of order $d-1$. Consider a scalar-flat metric from \cite{as} corresponding to an admissible parameter $\nu=(\alpha,\beta)$ such that
$$
\alpha>0, \quad \alpha-(\alpha+\beta)(d-1)>0.
$$
Then when $\nu=0$
$$
\int_X p=8(d-1)-\frac{8}{d-1}
$$
and when $\nu\ne0$ and $\alpha\ne 0$ then
$$
\int_X p=8(d-1)+4(\alpha+\beta)\left(\frac{1}{\alpha-(\alpha+\beta)(d-1)}-\frac{1}{\alpha}\right)
$$
\end{lemma}
\begin{proof}
From proposition \ref{finitep} we know that
$$
\int_X p= \sum_{i=1}^d\left[\frac{2V_H}{V^2}\right]_{-a_i}^{-a_{i-1}}+4\int_{r=0}\frac{\Det(g_H,g)dH}{V^2}.
$$
In our case as we have seen before 
$$
\nu_{i+1}-\nu_i=(1,-1)=v, \forall i=1,\cdots d
$$ 
thus
$$
g=\nu+\sum_{i=1}^{d-1}\frac{v}{2\rho_i}
$$
and 
$$
f=\nu_1+\sum_{i=1}^{d-1}\frac{1}{2}\left(1-\frac{H_i}{\rho_i}\right)v.
$$
We will start by calculating
$$
 \sum_{i=1}^d\left[\frac{V_H}{V^2}\right]_{-a_i}^{-a_{i-1}},
 $$
where $a_0=-\infty$ and $a_d=+\infty$. Now $V=\Det(g,f)$ and in this setting this can be simplified to yield
$$
\Det(g,f)=\Det(\nu,\nu_1)+\sum_{i=1}^{d-1}\frac{\Det(v,\nu_1)}{2\rho_i}+\sum_{i=1}^{d-1}\frac{\Det(\nu,v)}{2}\left(1-\frac{H_i}{\rho_i}\right)
$$
Now $\Det(\nu,\nu_1)=\alpha$ and $\Det(\nu,v)=-(\alpha+\beta)$ so that 
$$
\Det(g,f)=\alpha+\sum_{i=1}^{d-1}\frac{1}{2\rho_i}-(\alpha+\beta)\sum_{i=1}^{d-1}\frac{1}{2}\left(1-\frac{H_i}{\rho_i}\right)
$$
On $L_R$ we also have $V_H=\Det(g_H,f)$ because $g_r=0$. Now
$$
g_H=-\sum_{i=1}^{d-1}\frac{H_i}{2\rho_i^3}v
$$
therefore
$$
\Det(g_H,f)=-\sum_{i=1}^{d-1}\frac{H_i}{2\rho_i^3}
$$
When $H$ is close to $-a_i$ we have
$$
V=\frac{1}{2|H_i|}+O(0), \quad \Det(g_H,f)=-\frac{H_i}{2|H_i|^3}+O(0),
$$
where $O(0)$ is a bounded function of $H$ defined in a neighborhood of $-a_i$, so that 
$$
\lim_{H\rightarrow -a_i}\frac{V_H}{V^2}=-2\sg(H_i).
$$
The behavior of $\frac{V_H}{V^2}$ at infinity depends on weather $\nu=0$ or $\nu\ne 0$. Namely:
\begin{itemize}
\item when $\nu\ne 0$ at $+\infty$ we have
$$
\lim_{H\rightarrow +\infty}V=\alpha, \, V_H=-{}\frac{d-1}{2H^2}+O(0), \, \lim_{H\rightarrow +\infty}\frac{V_H}{V^2}=0;
$$
\item similarly when $\nu\ne 0$ and at $-\infty$
$$
\lim_{H\rightarrow +\infty}V=\alpha-(d-1)(\alpha+\beta), \, V_H={}\frac{d-1}{2H^2}+O(0),\, \lim_{H\rightarrow -\infty}\frac{V_H}{V^2}=0;
$$
\item When $\nu=0$ at $+\infty$ 
$$
V={}\frac{d-1}{2H}+O(0), \, V_H=-{}\frac{d-1}{2H^2}+O(0), \, \lim_{H\rightarrow \infty}\frac{V_H}{V^2}=-\frac{2}{d-1};
$$
\item Similarly when $\nu=0$ at $-\infty$ 
$$
V=-{}\frac{d-1}{2H}+O(0), \, V_H={}\frac{d-1}{2H^2}+O(0), \, \lim_{H\rightarrow -\infty}\frac{V_H}{V^2}=\frac{2}{d-1};
$$
\end{itemize}
Putting these together we see that for $i=2,\cdots d-1$
$$
 \left[\frac{V_H}{V^2}\right]_{-a_i}^{-a_{i-1}}=4.
 $$
When $\nu=0$ we also have
$$
 \left[\frac{V_H}{V^2}\right]_{-a_1}^{+\infty}=2-\frac{2}{d-1}, \quad \left[\frac{V_H}{V^2}\right]_{-\infty}^{-a_{d-1}}=2-\frac{2}{d-1},
$$
so that 
$$
\sum_{i=1}^d\left[\frac{V_H}{V^2}\right]_{-a_i}^{-a_{i-1}}=4(d-1)-\frac{4}{d-1}.
$$
In the case $\nu\ne 0$ we have instead that 
$$
 \left[\frac{V_H}{V^2}\right]_{-a_1}^{+\infty}=2, \quad \left[\frac{V_H}{V^2}\right]_{-\infty}^{-a_{d-1}}=2,
$$
so that 
$$
\sum_{i=1}^d\left[\frac{V_H}{V^2}\right]_{-a_i}^{-a_{i-1}}=4(d-1).
$$
On $L_R$ we have
$$
\frac{\Det(dg,g)}{\Det(g,f)^2}=\frac{\Det(g_H,g)}{V^2} 
$$
and we have already calculated 
$$
\lim_{R\rightarrow +\infty} \int_{L_R}\frac{\Det(dg,g)}{\Det(g,f)^2}
$$
in the proof of the previous lemma. Namely
$$
\lim_{R\rightarrow +\infty}\int_{L_R}\frac{\Det(g_H,g)dH}{V^2}=(\alpha+\beta)\left(\frac{1}{\alpha-(\alpha+\beta)(d-1)}-\frac{1}{\alpha}\right)
$$
\end{proof}
Putting the above two lemmas together  it is very easy to prove the main result in this section.
\begin{prop}
Let $X$ be the minimal toric resolution of a cyclic singularity of order $d-1$. Consider a scalar-flat metric from \cite{as} corresponding to an admissible parameter $\nu=(\alpha,\beta)$ such that
$$
\alpha>0, \quad \alpha-(\alpha+\beta)(d-1)>0.
$$
Then when $\nu=0$
$$
\int_X |\TR|^2=8(d-1)-\frac{8}{d-1}.
$$
When $\nu\ne 0$
$$
\int_X |\TR|^2=8(d-1)+3(\alpha+\beta)\left(\frac{1}{\alpha-(\alpha+\beta)(d-1)}-\frac{1}{\alpha}\right).
$$
\end{prop}

In the Ricci-flat case, which by \cite{as} corresponds to the case when $\alpha=-\beta$, these calculation agree with those the recent preprint \cite{al}.

\end{document}